\documentclass{amsart}
\usepackage[latin1]{inputenc}
\usepackage{amsmath}
\usepackage{amsthm}
\usepackage{amsfonts}
\usepackage{amssymb}
\usepackage[T1]{fontenc}
\usepackage{float}
\usepackage[all]{xy}
\usepackage{hyperref}

\newtheorem{theorem}{Theorem}[section]

\newtheorem{lem}[theorem]{Lemma}
\newtheorem{corol}[theorem]{Corollary}

\theoremstyle{definition}
\newtheorem{defi}[theorem]{Definition}

\newtheorem{rmq}[theorem]{Remark}

\def\Gr{{\rm{Gr}}}

\def\dim{{\rm{dim}\,}}
\def\ddim{{\mathbf{dim}\,}}

\def\<{\left<}
\def\>{\right>}

\def\ens#1{\left\{ #1 \right\}}
\def\fl{{\longrightarrow}\,}

\def\A{{\mathbb{A}}}

\def\P{{\mathbb{P}}}
\def\Q{{\mathbb{Q}}}
\def\Z{{\mathbb{Z}}}

\def\e{{\mathbf{e}}}
\def\k{{\mathbf{k}}}

\def\t{{\mathbf{t}}}
\def\x{{\mathbf{x}}}
\def\y{{\mathbf{y}}}
\def\kQ{{\mathbf{k}}Q}

\def\modd{{\textrm{mod-}}}

\def\ens#1{\left\{ #1 \right\}}
\def\Ext{{\rm{Ext}}}

\def\Hom{{\rm{Hom}}}

\def\d{{\partial}}

\def\im{{\rm{im}}}

\def\u{\underline}

\title[Positivity for Regular Cluster Characters]{Positivity for Regular Cluster Characters in Acyclic Cluster Algebras}
\author{G. Dupont}
\address{Universit\'e de Sherbrooke, D\'epartement de Math\'ematiques. 2500, Boul. de l'Universit\'e, J1K 2R1, Sherbrooke QC (CANADA).}
\email{gregoire.dupont@usherbrooke.ca}


\begin{document}

\begin{abstract}
	Let $Q$ be an acyclic quiver and let $\mathcal A(Q)$ be the corresponding cluster algebra. Let $H$ be the path algebra of $Q$ over an algebraically closed field and let $M$ be an indecomposable regular $H$-module. We prove the positivity of the cluster characters associated to $M$ expressed in the initial seed of $\mathcal A(Q)$ when either $H$ is tame and $M$ is any regular $H$-module, or $H$ is wild and $M$ is a regular Schur module which is not quasi-simple.
\end{abstract}

\maketitle

\setcounter{tocdepth}{1}
\tableofcontents

\section{Introduction}
	\subsection{Cluster algebras}
		Cluster algebras were introduced by Fomin and Zelevinsky in 2001 in order to design a combinatorial framework for studying total positivity in algebraic groups and canonical bases in quantum groups \cite{cluster1}. Since then, they have shown connections with various areas of mathematics like for instance combinatorics, Lie theory, Teichm\"uller theory, Poisson geometry, Donaldson-Thomas invariants and representation theory of algebras.

		A \emph{cluster algebra} $\mathcal A(Q,\x,\y)$ is defined from a \emph{seed} $(Q,\x,\y)$ where $Q$ is a quiver with $m \geq 1$ vertices and without loops or 2-cycles, $\x=(x_1, \ldots, x_m)$ is a $m$-tuple of indeterminates and $\y=(y_1, \ldots, y_m)$ is a $m$-tuple of elements of a semifield $\P$. 
		It is a $\Z\P$-subalgebra of the \emph{ambient field} $\Q\P(x_1, \ldots, x_m)$ equipped with a distinguished set of generators, called \emph{cluster variables}, gathered in possibly overlapping sets of fixed cardinality $m$, called \emph{clusters}. If the quiver $Q$ is without oriented cycles, the cluster algebra $\mathcal A(Q,\x,\y)$ is called \emph{acyclic}. If the semifield $\P$ is reduced to $\ens{1}$, the cluster algebra is called \emph{coefficient-free} and is denoted by $\mathcal A(Q,\x)$.

	\subsection{Positivity}
		In \cite{cluster1}, Fomin and Zelevinsky proved that every cluster variable in $\mathcal A(Q,\x,\y)$ can be written as a Laurent polynomial with coefficients in $\Z\P$ in any cluster, this is known as \emph{the Laurent phenomenon}. Moreover, they conjectured that Laurent expansions of cluster variables expressed in any cluster are subtraction-free, this is known as \emph{the positivity conjecture}. This conjecture was in particular established for the class of cluster algebras with a bipartite seed by Nakajima \cite{Nakajima:cluster} and the class of cluster algebras from surfaces by Musiker, Schiffler and Williams \cite{MSW:positivity} but it is still open in general. It is in particular open for acyclic cluster algebras. 

		If $Q$ is acyclic, it is well-known that the cluster algebra structure is strongly related to the structure of the category of representations of $Q$ over an algebraically closed field $\k$ and more precisely to the structure of the \emph{cluster category} introduced in \cite{BMRRT} (see for instance \cite{Keller:survey} and references therein). Nevertheless, for the sake of simplicity, we will only consider the category of representations of $Q$, which is enough for the purpose of this article. In this context, there is a map, called \emph{cluster character}, which associates to any representation $M$ of $Q$ a certain Laurent polynomial $X_M \in \Z\P[x_1^{\pm 1}, \ldots, x_m^{\pm 1}]$ \cite{CC,CK2,Palu,FK}. This Laurent polynomial is defined as a generating series for the Euler-Poincar\'e characteristics $\chi(\Gr_{\e}(M))$ of the quiver Grassmannians $\Gr_{\e}(M)$ where $\e$ runs over $\Z_{\geq 0}^n$ (see Section \ref{section:background} for definitions). It is known that this map realizes a bijection between the set of indecomposable \emph{rigid}, that is without self-extension, representations of $Q$ and the cluster variables of $\mathcal A(Q,\x,\y)$ which are distinct from the initial cluster variables $x_1, \ldots, x_m$ \cite{FK}. 

		Recently, Qin and Nakajima gave two independent proofs of the fact that if $M$ is a rigid representation of $Q$, then $\chi(\Gr_{\e}(M)) \geq 0$ \cite{Qin,Nakajima:cluster}. It proves in particular that cluster variables in $\mathcal A(Q,\x,\y)$ can be expressed as subtraction-free Laurent polynomials in the initial cluster of $\mathcal A(Q,\x,\y)$. 

		In this article, we will generalise this result to cluster characters associated to certain indecomposable non-rigid representations of the quiver $Q$. The motivation for this generalisation comes from the fact that such cluster characters arise in the constructions of cluster bases in cluster algebras.

	\subsection{Cluster bases}
		As we mentioned, the two main objectives in the theory of cluster algebras are the comprehension of total positivity in algebraic groups and of canonical bases in quantum groups. The problem of total positivity is related to the positivity conjecture we just discussed. The problem of canonical bases is related to the study of cluster bases we now present.

		Given a cluster algebra $\mathcal A(Q,\x,\y)$, a monomials in cluster variables belonging all to a same cluster is called a \emph{cluster monomial} in $\mathcal A(Q,\x,\y)$ and the set of cluster monomials is denoted by $\mathcal M_Q$. It is conjectured, and proved in several situations that cluster monomials are always independent over the ground ring $\Z\P$, see \cite{Plamondon:ClusterAlgebras}. A \emph{cluster basis} in $\mathcal A(Q,\x,\y)$ is thus defined as a $\Z\P$-linear basis of $\mathcal A(Q,\x,\y)$ containing the set of cluster monomials. The problem of constructing cluster bases was studied in \cite{shermanz,CK1,CZ,Cerulli:A21,Dupont:BaseAaffine,Dupont:transverse,DXX:basesv3,GLS:generic}. If the positivity conjecture holds, every cluster monomial can be written as a subtraction-free Laurent polynomial in any cluster of $\mathcal A(Q,\x,\y)$, such elements in $\mathcal A(Q,\x,\y)$ are called \emph{positive}. A natural problem is thus to construct cluster bases consisting of positive elements. A first step in this direction is to prove that known cluster bases consist of elements which can be written without subtraction in the initial cluster of $\mathcal A(Q,\x,\y)$.

		If $Q$ is a Dynkin quiver, known constructions of bases in $\mathcal A(Q,\x,\y)$ only involve cluster monomials \cite{CK1} and it is known that, in this case, cluster monomials can be realised in terms of cluster characters associated to rigid representations of a Dynkin quiver \cite{CC,FK}. Thus, it follows from the result by Qin and Nakajima that these bases can be expressed as subtraction-free Laurent polynomials in the initial cluster of $\mathcal A(Q,\x,\y)$.

		If $Q$ is acyclic and of infinite representation type, one can construct various cluster bases in $\mathcal A(Q,\x,\y)$ using the so-called \emph{generic bases} considered in \cite{Dupont:genericvariables,GLS:generic}. It is known that all these bases involve cluster characters associated to both rigid representations and indecomposable non-rigid representations of $Q$. Moreover, all the indecomposable representations considered in these constructions are \emph{Schur representations}, that is, have a trivial endomorphism ring \cite{Dupont:genericvariables}. Thus, the study of the positivity of the elements in the generic bases leads to the study of the positivity of the cluster characters associated to indecomposable non-rigid and in particular, regular Schur representations. 

	\subsection{Main results}
		In \cite{cluster4}, Fomin and Zelevinsky proved that the study of $\mathcal A(Q,\x,\y)$ over the ground ring $\Z\P$ associated to an arbitrary semifield $\P$ can be reduced to the case where $\P$ is the tropical semifield generated by the $m$-tuple $\y$. In this case, $\mathcal A(Q,\x,\y)$ is said to have \emph{principal coefficients at the initial seed $(Q,\x,\y)$}. Thus, without any other specifications, we always assume that $\P$ is the tropical semifield generated by $\y$ so that cluster variables and cluster characters belong to the ring $\Z[y_1, \ldots, y_m, x_1^{\pm 1}, \ldots, x_m^{\pm 1}]$.

		\begin{theorem}\label{theorem:wild}
			Let $Q$ be a representation-infinite acyclic quiver and let $M$ be a regular Schur $\kQ$-module which is not quasi-simple. Then $$X_M \in \Z_{\geq 0}[y_1, \ldots, y_m, x_1^{\pm 1}, \ldots, x_m^{\pm 1}].$$
		\end{theorem}

		As a corollary, we obtain the positivity of the Euler-Poincar\'e characteristics of the quiver Grassmannians for these modules~:
		\begin{corol}\label{corol:posGr}
			Let $Q$ be a representation-infinite acyclic quiver and let $M$ be a Schur regular $\kQ$-module which is not quasi-simple. Then for any dimension vector $\mathbf e \in \Z_{\geq 0}^m$ we have $\chi(\Gr_{\mathbf e}(M)) \geq 0$.
		\end{corol}
	
		\begin{rmq}
			Note that this result does not hold for any indecomposable representation of an acyclic quiver. Indeed, Derksen, Weyman and Zelevinsky provided an example of a quiver Grassmannian with negative Euler-Poincar\'e characteristics associated to an indecomposable non-rigid representation of a wild quiver \cite[Example 3.6]{DWZ:potentials2}. Nevertheless, our result does not contradict their example since the representation they consider is quasi-simple.
		\end{rmq}

		An essential ingredient in the proof of Theorem \ref{theorem:wild} is that Schur regular modules which are not quasi-simple have rigid quasi-composition factors. When $Q$ is an affine quiver, that is when $Q$ is of type $\widetilde {\mathbb A}, \widetilde {\mathbb D}$ or $\widetilde {\mathbb E}$, we can actually relax this assumption by using a certain combinatorial identity, called \emph{difference property}. This identity relates the cluster characters associated to modules in homogeneous tubes, which do not have rigid quasi-composition factors, to the cluster characters associated to modules in exceptional tubes, which have rigid quasi-composition factors. Nevertheless, the difference property was only established for cluster characters without coefficients \cite{Dupont:BaseAaffine,DXX:basesv3} so that the following theorem may only be proved for coefficient-free cluster characters, that is, for the specialisation $\u X_M = X_M|_{y_1 = \cdots = y_m=1}$.
		\begin{theorem}\label{theorem:tame}
			Let $Q$ be an affine quiver and let $M$ be a $\kQ$-module. Then $$\u X_M \in \Z_{\geq 0}[x_1^{\pm 1}, \ldots, x_m^{\pm 1}].$$
		\end{theorem}

		If $Q$ is an affine quiver, we also study three different cluster bases in the coefficient-free cluster algebras $\mathcal A(Q,\x)$. The first basis, which we denote by $\mathcal B_Q$, is a generalisation of the bases considered in \cite{shermanz,Cerulli:A21}. The second basis, denoted by $\mathcal C_Q$, is a generalisation of the basis considered in \cite{CZ} and the third basis $\mathcal G_Q$ is the generic basis considered in \cite{Dupont:genericvariables,GLS:generic}. For these cluster bases, we prove the following theorem~:
		\begin{theorem}\label{theorem:bases}
			Let $Q$ be an affine quiver, then~:
			\begin{enumerate}
				\item $\mathcal B_Q \subset \Z_{\geq 0}[x_1^{\pm 1}, \ldots, x_m^{\pm 1}]$~;
				\item $\mathcal C_Q \subset \Z_{\geq 0}[x_1^{\pm 1}, \ldots, x_m^{\pm 1}]$~;
				\item $\mathcal G_Q \subset \Z_{\geq 0}[x_1^{\pm 1}, \ldots, x_m^{\pm 1}]$.
			\end{enumerate}
		\end{theorem}
		
	\subsection{Organisation of the paper}
		In Section \ref{section:background}, we fix notations and recall the background concerning cluster characters and the structure of regular modules over representation-infinite hereditary algebras. In Section \ref{section:gCheb}, we study \emph{generalised Chebyshev polynomials} whose combinatorics are closely connected to those of cluster characters in regular components. In Section \ref{section:wild} we use these polynomials to prove Theorem \ref{theorem:wild} and Corollary \ref{corol:posGr}. In Section \ref{section:Delta} we study another family of multivariate polynomials, called \emph{$\Delta$-polynomials}, which are used to prove Theorem \ref{theorem:tame} in Section \ref{section:tame}. Finally, Section \ref{section:bases} is devoted to the precise statement and the proof of Theorem \ref{theorem:bases}.

\section{Notations and background}\label{section:background}
	Let $Q$ denote a finite connected quiver which is acyclic and of infinite representation type. The set of vertices in $Q$ is denoted by $Q_0$ and the cardinality of this set is denoted by $m$. We fix a $Q_0$-tuple of indeterminates $\x=(x_i | i \in Q_0)$, a semifield $\P$ and a $Q_0$-tuple $\y=(y_i\ | \ i \in Q_0)$ of elements of $\P$ such that $\P$ is the tropical semifield generated by $y_1, \ldots, y_m$. We denote by $\mathcal A(Q,\x,\y)$ the cluster algebra with principal coefficients at the initial seed $(Q,\x,\y)$.

	\subsection{Cluster characters}
		Let $\k$ be an algebraically closed field. We denote by $\kQ$ the path algebra of $Q$ over $\k$. We naturally identify the category of (finite dimensional) representations of $Q$ over $\k$ with the category mod-$\kQ$ of finitely generated right-$\kQ$-modules. For any vertex $i \in Q_0$, we denote by $S_i$ the simple representation of $Q$ at $i$. Given a representation $M$ of $Q$, we denote by $\ddim M \in \Z^{Q_0}_{\geq 0}$ its \emph{dimension vector}. We denote by $\<-,-\>$ the \emph{Euler form} on mod-$\kQ$ given by 
		$$\<M,N\> = \dim \Hom_{\kQ}(M,N) - \dim \Ext^1_{\kQ}(M,N) \in \Z$$
		for any $M,N$ in mod-$\kQ$. It is well-known that this form only depends on the dimension vectors.

		For any representation $M$ of $Q$ and any dimension vector $\e \in \Z^{Q_0}_{\geq 0}$, we set
		$$\Gr_{\e} = \ens{N \textrm{ subrepresentation of } M \ | \ \ddim N = \e}.$$
		It is a projective variety, called the \emph{quiver Grassmannian} of $M$ (of dimension $\e$), and we denote by $\chi(\Gr_{\e}(M))$ its Euler-Poincar\'e characteristic (with respect to the simplicial cohomology if $\k$ is the field of complex numbers and to the compactly supported \'etale cohomology if $\k$ is arbitrary). 

		For any representation $M$ of $Q$, the \emph{cluster character associated to $M$} is 
		$$X_M = \sum_{\e \in \Z^{Q_0}_{\geq 0}} \chi(\Gr_{\e}(M)) \prod_{i \in Q_0} y_i^{e_i} x_i^{-\<\e,\ddim S_i\>-\<\ddim S_i, \ddim M - \e\>} \in \Z[\y,\x^{\pm 1}]$$
		where we use the short-hand notation $\Z[\y,\x^{\pm 1}] = \Z[y_1, \ldots, y_m, x_1^{\pm 1}, \ldots, x_m^{\pm 1}]$.

		The \emph{coefficient-free cluster character associated to $M$} is the specialisation of $X_M$ at $y_i=1$ for any $1 \leq i \leq m$, that is,
		$$\u X_M = \sum_{\e \in \Z^{Q_0}_{\geq 0}} \chi(\Gr_{\e}(M)) \prod_{i \in Q_0} x_i^{-\<\e,\ddim S_i\>-\<\ddim S_i, \ddim M - \e\>} \in \Z[\x^{\pm 1}].$$

		This map was first introduced in \cite{CC} and it was later generalised in \cite{CK2,Palu,FK,DWZ:potentials2,Plamondon:clusterchar}. It is known that the cluster character (resp. the coefficient-free cluster character) induces a bijection from the set of indecomposable rigid representations of $Q$ to the set of cluster variables in $\mathcal A(Q,\x,\y)$ (resp. $\mathcal A(Q,\x)$) distinct from $x_1, \ldots, x_m$ \cite{CK2,FK}.

	\subsection{Regular modules over representation-infinite hereditary algebras}
		Throughout the article, we will make a free use of classical results on the structure of regular modules over a representation-infinite hereditary algebra. For details concerning this theory, we refer for instance the reader to \cite{Ringel:1099} for affine quivers and to \cite{Kerner:repwild} for wild quivers.

		We denote by $\tau$ the Auslander-Reiten translation on mod-$\kQ$. An indecomposable $\kQ$-module is called \emph{regular} if it is neither in the $\tau^{-1}$-orbit of a projective module nor in the $\tau$-orbit of an injective module. The Auslander-Reiten quiver of mod-$\kQ$ is denoted by $\Gamma(\modd\kQ)$. The components of $\Gamma(\modd \kQ)$ containing neither projective modules nor injective modules are called \emph{regular components} of $\Gamma(\modd\kQ)$ and it is known that every indecomposable regular module belongs to a regular component.

		Every regular component $\mathcal R$ in $\Gamma(\modd\kQ)$ is of the form $\Z\A_\infty/(\tau^p)$ for some $p \geq 0$ (see for instance \cite[Section VIII.4]{ARS}). If $p \geq 1$, $\mathcal R$ is called a \emph{tube of rank $p$}. If $p=1$, the tube is called \emph{homogeneous}, if $p>1$, the tube is called \emph{exceptional}. If $p=0$, $\mathcal R$ is called a \emph{sheet}.

		Let $\mathcal R$ be a regular component in $\Gamma(\modd\kQ)$ and $M$ be an indecomposable object in $\mathcal R$. Then there exists a unique family $\ens{R^{(0)},R^{(1)} \ldots, R^{(n)}}$ of indecomposable $\kQ$-modules in $\mathcal R$ such that $R^{(0)}=0$ and such that there is a sequence of irreducible monomorphisms
		$$R^{(1)} \fl R^{(2)} \fl \cdots \fl R^{(n)}=M.$$
		$R^ {(1)}$ is called the \emph{quasi-socle} of $M$ and the integer $n$ is called the \emph{quasi-length} of $M$. The module $M$ is called \emph{quasi-simple} if its quasi-length is equal to 1. The quotients $M_i=R^{(i)}/R^{(i-1)}, i=1, \ldots, n$ are called the \emph{quasi-composition factors} of $M$, with the convention that $R^{(0)}=0$. The modules of quasi-length 1 are called \emph{quasi-simple} in $\mathcal R$. Note that $\tau M_i \simeq M_{i-1}$ for every $i=2, \ldots, n$. 

		If $Q$ is an affine quiver, the regular components form a $\P^1(\k)$-family of tubes and at most three of these tubes are exceptional. If $Q$ is wild, every regular component is a sheet. It is well-known that quasi-simple $\kQ$-modules in exceptional (resp. homogeneous) tubes are always (resp. never) rigid and that quasi-simple modules in sheets can be rigid or not.

\section{Generalised Chebyshev polynomials}\label{section:gCheb}
	As it was observed in \cite{Dupont:stabletubes,Dupont:qChebyshev}, the combinatorics of cluster characters associated to regular modules are governed by a family of multivariate polynomials. The aim of this section is to give prove some technical properties of these polynomials which will be useful in the proofs of Theorem \ref{theorem:wild} and Theorem \ref{theorem:tame}.

	Let $q_i,t_i$ with $i \in \Z$ be indeterminates over $\Z$. We set $\mathbf q = \ens{q_i\ | \ i \in \Z}$ and $\mathbf t = \ens{t_i\ | \ i \in \Z}$ and for any subset $J \in \Z$ we set $\mathbf q_J= \ens{q_i\ | \ i \in J}$ and $\mathbf t_J= \ens{q_i\ | \ i \in J}$. For every $n \geq 1$, the \emph{$n$-th quantized generalized Chebyshev polynomial} is the polynomial $P_n(q_1, \ldots, q_n, t_1, \ldots, t_n) \in \Z[\mathbf q_{[1,n]}, \mathbf t_{[1,n]}]$ given by 
	\begin{equation}\label{eq:detPn}
		P_n(q_1, \ldots, q_n,t_1, \ldots, t_n)
		=\det \left[\begin{array}{ccccccc}
			t_n & 1 &&& (0)\\
			q_n & \ddots & \ddots \\
			& \ddots & \ddots & \ddots \\
			& & \ddots & \ddots & 1 \\
			(0)& & & q_2 & t_1
		\end{array}\right].
	\end{equation}
	Note that $P_n$ does not depend on the variable $q_1$ so that it can be viewed as a polynomial in the $2n-1$ variables $q_2, \ldots, q_n, t_1, \ldots, t_n$. Nevertheless, it is usually more convenient to consider $P_n$ as a polynomial in the $2n$ variables $q_1, \ldots, q_n,t_1, \ldots, t_n$.

	\subsection{Partial derivatives of generalised Chebyshev polynomials}
		We first prove that generalised Chebyshev polynomials satisfy simple partial differential equations. 
		\begin{lem}\label{lem:dpSn}
			For every $n \geq 1$ and $1 \leq i \leq n$, we have
			$$\frac{\d}{\d t_i}P_n(\mathbf q_{[1, n]},\mathbf t_{[1, n]})=P_{i-1}(\mathbf q_{[1,i-1]},\mathbf t_{[1,i-1]})P_{n-i}(\mathbf q_{[i+1,n]},\mathbf t_{[i+1,n]})$$
		\end{lem}
		\begin{proof}
			Expanding the determinant with respect to the $i$-th column in equation \eqref{eq:detPn}, we get
			$$P_n(\mathbf q_{[1,n]}, \mathbf t_{[1,n]})=t_iP_{i-1}(\mathbf q_{[1,i-1]},\mathbf t_{[1,i-1]})P_{n-i}(\mathbf q_{[i+1,n]},\mathbf t_{[i+1,n]})+Q(\mathbf q,\mathbf t)$$
			where $Q$ does not depend on the variable $t_i$. Thus, taking the derivative in $t_i$, we get 
			$$\frac{\d}{\d t_i}P_n(\mathbf q_{[1, n]},\mathbf t_{[1, n]})=P_{i-1}(\mathbf q_{[1,i-1]},\mathbf t_{[1,i-1]})P_{n-i}(\mathbf q_{[i+1,n]},\mathbf t_{[i+1,n]}).$$
		\end{proof}

	\subsection{Positivity properties of generalised Chebyshev polynomials}
		\begin{lem}\label{lem:CC}
			Let $\mathbf q=\ens{q_i|i \in \Z}$, $\mathbf t=\ens{t_i|i \in \Z}$ be families of indeterminates over $\Z$. Then for any $n \geq 1$, 
			$$P_n(q_1, \ldots, q_n, t_1 + \frac{q_1}{t_2}, \ldots, t_n + \frac{q_n}{t_{n+1}}) \in \Z_{\geq 0}[\mathbf q_{[1,n]},\mathbf t_{[1,n+1]}^{\pm 1}]$$
		\end{lem}
		\begin{proof}
			Let $n \geq 1$ be an integer. Consider the quiver $A$ of Dynkin type $\A_{n+2}$ equipped with the following orientation~:
			$$\xymatrix{
				n+1 \ar[r] & n \ar[r] & \cdots \ar[r] & 1 \ar[r] & 0
			}$$
			whose simple modules are denoted by $\Sigma_0, \ldots, \Sigma_{n+1}$. For any $l \geq 1$ and any $0 \leq i \leq n+1$ we denote by $\Sigma_i^{(l)}$ the unique indecomposable representation of $A$ with socle $\Sigma_i$ and length $l$, whenever it exists. Denote by $X^A_? : M \mapsto X^A_M \in \Z[y_0, \ldots, y_{n+1}, x_0^{\pm 1}, \ldots, x_{n+1}^{\pm 1}]$ the cluster character (with principal coefficients) associated to this quiver. Then it is well-known that $X^A_{\Sigma_i} = \frac{1}{x_i}(x_{i-1}+y_ix_{i+1})$ for any $1 \leq i \leq n$. Consider thus the injective homomorphism of $\Z$-algebras given by
			$$\phi: \left\{\begin{array}{rcl}
				\Z[\mathbf q_{[1,n]}, \t_{[1,n+1]}^{\pm 1}] & \fl & \Z[\y_{[0,n+1]}, \x_{[0,n+1]}^{\pm 1}]\\
				q_i & \mapsto & y_i \\
				t_i & \mapsto & \frac{x_{i-1}}{x_i}
			\end{array}\right.$$
			and note that this isomorphism preserve the positive cones in the sense that $\phi(\Z_{\geq 0}[\mathbf q_{[1,n]}, \t_{[1,n+1]}^{\pm 1}]) = \Z_{\geq 0}[\mathbf y_{[0,n+1]}, \x_{[0,n+1]}^{\pm 1}] \cap \im(\phi)$.

			Then we have 
			$$\phi(t_i+\frac{q_i}{t_{i+1}}) = \frac{x_{i-1}}{x_i} + y_i \frac{x_{i+1}}{x_i} = X^A_{\Sigma_i}$$
			so that 
			
			$$
			\phi(P_n(q_1, \ldots,q_n, t_1 + \frac{q_1}{t_2}, \ldots, t_n + \frac{q_n}{t_{n+1}})) = P_n(y_1, \ldots, y_n, X^A_{\Sigma_1}, \ldots, X^A_{\Sigma_n}) = X^A_{\Sigma^{(n)}}
			$$
			where the last equality follows from \cite[Theorem 3]{Dupont:qChebyshev}. In particular, since $\Sigma^{(n)}$ is an indecomposable representation of a quiver of Dynkin type $A$, its quiver Grassmannians have non-negative characteristics and it follows that $X^A_{\Sigma^{(n)}} \in \Z_{\geq 0}[\mathbf y_{[0,n+1]}, \x_{[0,n+1]}^{\pm 1}]$ so that $P_n(q_1, \ldots, q_n, t_1 + \frac{q_1}{t_2}, \ldots, t_n + \frac{q_n}{t_{n+1}}) \in \Z_{\geq 0}[\mathbf q_{[1,n]}, \t_{[1,n+1]}^{\pm 1}]$.
		\end{proof}

		We now prove a technical lemma.
		\begin{lem}\label{lem:Pnpos}
			Let $\mathbf q=\ens{q_i|i \in \Z}$, $\mathbf t=\ens{t_i|i \in \Z}$ and $\mathbf u=\ens{u_i|i \in \Z}$ be families of indeterminates over $\Z$. Then for any $n \geq 1$, 
			$$P_n(q_1, \ldots, q_n, t_1 + u_1 + \frac{q_1}{t_2}, \ldots, t_n + u_n + \frac{q_n}{t_{n+1}}) \in \Z_{\geq 0}[\mathbf q_{[1,n]},\mathbf u_{[1,n]},\mathbf t_{[1,n+1]}^{\pm 1}].$$
		\end{lem}
		\begin{proof}
			Consider the polynomial $P_n(q_1, \ldots, q_n, t_1 + u_1 + \frac{q_1}{t_2}, \ldots, t_n + u_n + \frac{q_n}{t_{n+1}})$ viewed as an element of $\Z[q_1, \ldots, q_n, t_1^{\pm 1}, \ldots, t_{n+1}^{\pm 1}][u_1, \ldots, u_n]$. Taking the Taylor expansion, we get that $P_n(q_1, \ldots, q_n, t_1 + u_1 + \frac{q_1}{t_2}, \ldots, t_n + u_n + \frac{q_n}{t_{n+1}})$ equals
			$$\sum_{i_1, \ldots, i_n} \frac{\d^{i_1} \cdots \d^{i_n}}{\d t_1\cdots \d t_n}P_n(q_1, \ldots, q_n, t_1 + \frac{q_1}{t_2}, \ldots, t_n + \frac{q_n}{t_{n+1}})u_1^{i_1}\cdots u_n^{i_n}$$
 			But according to Lemma \ref{lem:dpSn}, every partial derivative $\sum_{i_1, \ldots, i_n} \frac{\d^{i_1} \cdots \d^{i_n}}{\d t_1\cdots \d t_n}P_n(q_1, \ldots, q_n, t_1+\frac{q_1}{t_2}, \ldots, t_n+\frac{q_n}{t_{n+1}})$ is a product of $P_j(q_k, \ldots, q_{k+j-1},t_k+\frac{q_k}{t_{k+1}}, \ldots, t_{k+j-1}+\frac{q_{k+j-1}}{t_{k+j}})$ for some $j<n$. According to Lemma \ref{lem:CC}, each of these $P_j(q_k, \ldots, q_{k+j-1},t_{k}+\frac{q_k}{t_{k+1}}, \ldots, t_{k+j-1}+\frac{q_{k+j-1}}{t_{k+j-1}})$ is in $\Z_{\geq 0}[\mathbf q_{[1,n]},\mathbf t_{[1,n+1]}^{\pm 1}]$ and thus each partial derivative is in $\Z_{\geq 0}[\mathbf q_{[1,n]},\mathbf t_{[1,n+1]}^{\pm 1}]$. The lemma follows from the Taylor expansion. 
		\end{proof}

\section{Proof of Theorem \ref{theorem:wild} and Corollary \ref{corol:posGr}}\label{section:wild}
	We now prove Theorem \ref{theorem:wild}. We thus fix an acyclic quiver $Q$ and for any $\kQ$-module $M$ and any dimension vector $\e \in \Z_{\geq 0}^{Q_0}$, we set 
	$$L(M, {\e}) = \chi(\Gr_{\e}(M)) \y^\e \prod_{i \in Q_0} x_i^{-\<\e, \ddim S_i \> - \<\ddim S_i, \ddim M - \e\>}$$
	where we adopt the notation $\y^\e=\prod_{i \in Q_0} y_i^{e_i}$.
	Thus, we have 
	$$X_M = \sum_{\e \in \Z_{\geq 0}^{Q_0}} L(M,{\e}).$$

	The following lemma will be essential in the following~:
	\begin{lem}\label{lem:key}
		Let $Q$ be an acyclic quiver and let $M$ be an indecomposable non-projective $\kQ$-module, then 
		$$L(M,0) = \frac{\y^{\ddim \tau M}}{L(\tau M,\ddim \tau M)}.$$
	\end{lem}
	\begin{proof}
		We have
		\begin{align*}
			L(M,0) 
				& = \chi(\Gr_{0}(M)) \prod_{i \in Q_0} x_i^{-\<\ddim S_i, \ddim M\>} \\
				& = \prod_{i \in Q_0} x_i^{-\<\ddim S_i, \ddim M\>} \\
		\end{align*}
		and 
		\begin{align*}
			L(\tau M,\ddim \tau M) 
				& = \chi(\Gr_{\ddim \tau M}(\tau M)) \y^{\ddim \tau M} \prod_{i \in Q_0} x_i^{-\<\ddim \tau M, \ddim S_i\>} \\
				& = \y^{\ddim \tau M} \prod_{i \in Q_0} x_i^{-\<\ddim \tau M, \ddim S_i\>} \\
				& = \y^{\ddim \tau M} \prod_{i \in Q_0} x_i^{\<\ddim S_i, \ddim M\>} \\
				& = \y^{\ddim \tau M} L(M,0)^{-1}.
		\end{align*}
	\end{proof}

	\subsection{Proof of Theorem \ref{theorem:wild}}
		We now assume that $Q$ is representation-infinite. We fix an indecomposable regular $\kQ$-module $M$ and we denote by $\mathcal R$ the regular component containing $M$ which is of the form $\mathcal R \simeq \Z \A_{\infty}/(\tau^p)$ for some $p \geq 0$. We denote by $R_{i}$ with $i \in \Z/p\Z$ the quasi-simple modules in $\mathcal R$ and we order them in such a way that $\tau R_i \simeq R_{i-1}$ for any $i \in \Z/p\Z$. Without loss of generality, we can assume that $M \simeq R_1^{(n)}$ for some $n \geq 1$. 

		Then it follows from \cite[Theorem 2]{Dupont:qChebyshev} that 
		\begin{equation}\label{eq:XRin}
			X_M = P_n(\y^{\ddim R_1}, \ldots, \y^{\ddim R_n}, X_{R_1}, \ldots, X_{R_n})
		\end{equation}

		For any $i \in \Z/p\Z$ we set 
		$$\tau_i = L(R_i,0) \textrm{ and } \nu_i = \sum_{\mathbf e \neq \ddim R_i, 0} L(R_i,\mathbf e).$$
		With these notations, it follows from Lemma \ref{lem:key} that 
		$$X_{R_i} = \tau_i + \nu_i + \y^{\ddim R_i} \frac{1}{\tau_{i+1}}$$
		for any $i \in \Z/p\Z$.

		Replacing in \eqref{eq:XRin}, we get 
		\begin{equation}\label{eq:XMtau}
			X_M = P_n(\y^{\ddim R_1}, \ldots, \y^{\ddim R_n}, \tau_1 + \nu_1 + \y^{\ddim R_1} \frac{1}{\tau_{2}}, \ldots, \tau_n + \nu_n + \y^{\ddim R_i} \frac{1}{\tau_{i+n}}) 
		\end{equation}
		and thus specialising Lemma \ref{lem:Pnpos} at $q_i= \y^{\ddim R_i}$, $u_i=\nu_i$ and $t_i=\tau_i$, we get $$X_M \in \Z_{\geq 0}[\y^{\ddim R_i}, \tau_j^{\pm 1}, \nu_i | 1 \leq i \leq n, 1 \leq j \leq n+1]$$
		
		Assume now that $n \geq 2$ and that $M$ is a Schur module, then it follows from \cite[\S 9.2]{Kerner:repwild} that the quasi-composition factors $R_1, \ldots, R_n$ of $M$ are rigid. In particular, $\chi(\Gr_{\e}(R_i)) \geq 0$ for any dimension vector $\e \in \Z_{\geq 0}^{Q_0}$ and any $1 \leq i \leq n$. Thus $L(R_i,\mathbf e) \in \Z_{\geq 0}[\y, \x^{\pm 1}]$ for any $1 \leq i \leq n$. Since each $\tau_i$ is a monomial in $\Z[\y,\x^{\pm 1}]$, we get $\tau_i^{\pm 1} \in \Z_{\geq 0}[\y, \x^{\pm 1}]$ for any $1 \leq i \leq n+1$. Now, since $\Z_{\geq 0}[\y, \x^{\pm 1}]$ is a semiring, we have $\Z_{\geq 0}[\y^{\ddim R_i}, \tau_i^{\pm 1}, \nu_i | 1 \leq i \leq n+1] \subset \Z_{\geq 0}[\y, \x^{\pm 1}]$. Thus it follows from equation \eqref{eq:XMtau} that $X_M \in \Z_{\geq 0}[\y, \x^{\pm 1}]$. This proves the theorem. \qed

	\subsection{Proof of Corollary \ref{corol:posGr}}
		Let $Q$ be a representation-infinite acyclic quiver algebra and let $M$ be a regular Schur $\kQ$-module which is not quasi-simple. According to Theorem \ref{theorem:wild}, we have $X_M \in \Z_{\geq 0}[\y, \x^{\pm 1}]$ and thus, specialising the $x_i$'s at 1, we get
		$$\sum_{\e \in \Z^{Q_0}_{\geq 0}} \chi(\Gr_{\e}(M)) \prod_{i \in Q_0} y_i^{e_i} \in \Z_{\geq 0}[\y].$$
		The ring $\Z_{\geq 0}[\y]$ is naturally $\Z^{Q_0}$-graded by setting $\textrm{deg}(y_i)$ to be the $i$-th vector of the canonical basis of $\Z^{Q_0}$ for any $i \in Q_0$. Thus, identifying the graded components, we get $\chi(\Gr_{\e}(M)) \in \Z_{\geq 0}$ for any $\e \in \Z^{Q_0}_{\geq 0}$. This proves the corollary. \qed

\section{$\Delta$-polynomials}\label{section:Delta}
	Generalised Chebyshev polynomials allow to express cluster characters associated to regular modules in terms of cluster characters associated to their quasi-composition factors. As it appeared in the proof of Theorem \ref{theorem:wild}, when the quasi-composition factors are rigid, the positivity of the considered character can be deduced from the positivity of the characters associated to the quasi-composition factors. In the tame case, regular modules in homogeneous tubes have non-rigid quasi-composition factors. Nevertheless, the $\Delta$-polynomials we now introduce allow to express cluster characters associated to regular modules in homogeneous tubes in terms of cluster characters associated to quasi-simple modules in exceptional tubes, which are known to be rigid. These $\Delta$-polynomials actually come from difference properties introduced in \cite{Dupont:BaseAaffine,DXX:basesv3,Dupont:transverse}.
 		
 		\begin{defi}
			Let $q_i,t_i$ with $i \in \Z$ be indeterminates over $\Z$. For any $p \geq 1$, and $l \geq 1$, we set
			$$\Delta_{l,p}(\mathbf q_{[1,lp]}, \mathbf t_{[1,lp]})=P_{lp}(\mathbf q_{[1,lp]}, \mathbf t_{[1,lp]})-q_1P_{lp-2}(\mathbf q_{[2,lp-1]}, \mathbf t_{[2,lp-1]})$$
			and 
			$$\u \Delta_{l,p}(\mathbf t_{[1,lp]}) = \Delta_{l,p}(1, \ldots, 1, \mathbf t_{[1,lp]})$$
		\end{defi}

		\begin{rmq}
			Note that we will only consider the ``coefficient-free'' $\Delta$-polynomials of the form $\u \Delta_{l,p}$ in Sections \ref{section:tame} and \ref{section:bases} since the difference properties are only known for cluster characters without coefficients. Nevertheless, it seems that the $\Delta$-polynomials of the form $\Delta_{l,p}$ are the polynomials to consider for cluster characters with coefficients. This is the reason why we do not restrict to polynomials of the form $\u \Delta_{l,p}$ in this section.
		\end{rmq}

		We now prove the analogue of Lemma \ref{lem:Pnpos} for $\Delta$-polynomials~: 
		\begin{lem}\label{lem:positivitedeltap}
			Let $q_i,t_i,u_i$ with $i \in \Z$ be indeterminates. Then for any $l,p \geq 1$, we have
			$$\Delta_{l,p}(q_1, \ldots, q_{lp}, t_1+u_1+\frac{q_1}{t_2}, \ldots, t_{lp}+u_{lp}+\frac{q_{lp}}{t_1})\in \Z_{\geq 0}[q_i, t_i^{\pm 1},u_i \ | \ 1 \leq i \leq lp].$$
		\end{lem}
		\begin{proof}
			As for Lemma \ref{lem:Pnpos}, the proof is based on a Taylor expansion of $\Delta_{l,p}$. In order to shorten notations, for every $n \geq 1$ and any $i \in\Z$, we set
			$$P_n([i,i+n-1])=P_n(\mathbf q_{[i,i+n-1]}, \mathbf t_{[i,i+n-1]})$$
			and
			$$\Delta_{l,p}([i,i+lp-1])=\Delta_{l,p}(\mathbf q_{[i,i+lp-1]}, \mathbf t_{[i,i+lp-1]}).$$
			
			Let $p\geq 2$ be an integer, we claim that for every $i=1, \ldots, p$, 
			$$\frac{\d}{\d t_i}\Delta_{l,p}([1,p])=P_{lp-1}(q_{i+1}, \ldots, q_{lp}, q_1, \ldots, q_{i-1}, t_{i+1}, \ldots, t_{lp}, t_1, \ldots, t_{i-1}).$$
			
			Indeed, we have
			$$\frac{\d}{\d t_i}\Delta_{l,p}([1,lp])
			 		=\frac{\d}{\d t_i}P_{lp}([1,lp])-q_1\frac{\d}{\d t_i}P_{lp-2}([2,lp-1]).$$
			If $i=1$, the claim clearly holds. Now if $i>1$, using Lemma \ref{lem:dpSn} and the following three terms relations for generalized Chebyshev polynomials (obtained by expanding the determinant expression of $P_{i-1}$)
			$$P_{i-1}([1,i-1])=t_1P_{i-2}([2,i-1])-q_2P_{i-3}([3,i-1]),$$
			we get~:
			\begin{align*}
			 	\frac{\d}{\d t_i}\Delta_{l,p}([1,p])
			 		& =-q_2 P_{i-3}([3,i-1])P_{lp-i}([i+1,lp])\\
			 		& +t_1 P_{i-2}([2,i-1])P_{lp-i}([i+1,lp])\\
			 		& -q_1 P_{i-2}([2,i-1])P_{lp-i-1}(i+1,lp-1)
			\end{align*}
			On the other hand, $P_{lp-1}(q_{i+1}, \ldots, q_{lp}, q_1, \ldots, q_{i-1}, t_{i+1}, \ldots, t_{lp}, t_1, \ldots, t_{i-1})$ is~:
			$$D=\det \left[\begin{array}{cccccccc}
				t_{i-1} & 1 & & & & & & (0)\\
				q_{i-1} & \ddots & \ddots & & & \\
				 & \ddots & \ddots & 1 & & \\
				 & & \ddots & t_1 & 1\\ 
				 & & & q_1 & t_{lp} & \ddots \\
				 & & & & q_{lp} & \ddots & \ddots \\
				 & & & & & \ddots & \ddots & 1 \\
				 (0) & & & & & & q_{i+2} & t_{i+1}
			\end{array}\right]$$
			Expanding with respect to the column containing $q_1$, we get:
			$$D=(-1)A+t_1 B -q_1 C$$
			where $A,B,C$ are $n-1 \times n-1$ minors. Computing these minors, we get:
			$$A=q_2P_{i-3}([3,i-1])P_{lp-i}([i+1,lp]) ~;$$
			$$B=P_{i-2}([2,i-1])P_{lp-i}([i+1,lp]) ~;$$
			$$C=P_{i-2}([2,i-1])P_{lp-i-1}([i+1,lp-1]).$$
			Hence
			$$\frac{\d}{\d t_i}\Delta_{l,p}([1,lp])=P_{lp-1}(q_{i+1}, \ldots, q_{lp}, q_1, \ldots, q_{i-1}, t_{i+1}, \ldots, t_{lp}, t_1, \ldots, t_{i-1})$$
			and the claim is proved.
			
			Considering the Taylor expansion in $(u_1, \ldots, u_{lp})$, we get that $\Delta_{l,p}(q_1, \ldots, q_{lp}, t_1+u_1+\frac{q_1}{t_2}, \ldots, t_{lp}+u_{lp}+\frac{q_{lp}}{t_{1}})$ equals
			$$\sum_{i_1, \ldots, i_p}\left( \frac{\d^{i_1} \cdots  \d^{i_p}}{\d{t_1}\cdots\d{t_p}} \right)\Delta_{l,p}(q_1, \ldots, q_{lp}, t_1+\frac{q_1}{t_2}, \ldots, t_{lp}+\frac{q_{lp}}{t_1})u_1^{i_1}\cdots u_{p}^{i_p}
			$$
			But for every $i_1, \ldots, i_p$, it follows from the claim that 
			$$\left(\frac{\d^{i_1} \cdots  \d^{i_p}}{\d{t_1}\cdots\d{t_p}} \right)\Delta_{l,p}(q_1, \ldots, q_{lp},t_1+\frac{q_1}{t_2}, \ldots, t_{lp}+\frac{q_{lp}}{t_1})$$
			is a derivative of a generalized Chebyshev polynomial in consecutive variables. Thus, by Lemma \ref{lem:Pnpos}, this is again a product of generalized Chebyshev polynomial in consecutive variables. Then, Lemma \ref{lem:Pnpos} implies that 
			$$\left(\frac{\d^{i_1} \cdots  \d^{i_p}}{\d{t_1}\cdots\d{t_p}} \right)\Delta_{l,p}(q_1, \ldots, q_{lp},t_1+\frac{q_1}{t_2}, \ldots, t_{lp}+\frac{q_{lp}}{t_1}) \in \Z_{\geq 0}[q_i, t_i^{\pm 1} \ | \ 1 \leq i \leq lp].$$
			so that finally,
			$$\Delta_{l,p}(q_1, \ldots, q_{lp}, t_1+u_1+\frac{q_1}{t_2}, \ldots, t_{lp}+u_{lp}+\frac{q_{lp}}{t_{1}}) \in \Z_{\geq 0}[q_i, t_i^{\pm 1},u_i \ | \ 1 \leq i \leq lp]$$
			and the lemma is proved.
		\end{proof}
 		
 		\begin{rmq}
			Note that the positivity of $\Delta_{l,p}(q_1, \ldots, q_p, t_1+u_+\frac{q_1}{t_2}, \ldots, t_{lp}+u_{lp}+\frac{q_{lp}}{t_{1}})$ really comes from the fact that in the last variable $t_1$ occurs instead of $t_{lp+1}$. This is an illustration of the fact that the $\Delta$-polynomials arise from tubes, that is, periodic (in this case, $p$-periodic) regular components. Indeed, in general, the polynomial $\Delta_{l,p}(q_1, \ldots, q_{lp}, t_1+u_1+\frac{q_1}{t_2}, \ldots, t_{lp}+u_{lp}+\frac{q_{lp}}{t_{lp+1}})$ is not subtraction-free. Consider for instance $l=1,p=2$ and specialise the $u_i$'s to 1, then 
			\begin{align*}
				\Delta_{1,2}(q_1, q_2, t_1+\frac{q_1}{t_2}, t_2+\frac{q_2}{t_1})
					&=\frac{t_1^2t_2^2+q_1q_2}{t_2t_1}\in  \Z_{\geq 0}[q_1,q_2,t_1^{\pm 1},t_2^{\pm 1}]\\
			\end{align*}
			whereas
			\begin{align*}
				\Delta_{1,2}(q_1, q_2, t_1+\frac{q_1}{t_2}, t_2+\frac{q_2}{t_3})
					&=\frac{t_1t_2^2t_3+q_2(t_1t_2-t_2t_3)+q_1q_2}{t_2t_3}\\
					& \not \in \Z_{\geq 0}[q_1,q_2,t_1^{\pm 1},t_2^{\pm 1},t_3^{\pm 1}].\\
			\end{align*}
		\end{rmq}

\section{Proof of Theorem \ref{theorem:tame}}\label{section:tame}
	Let $Q$ be an affine quiver. In this section we prove that $\u X_M \in \Z[\x^{\pm 1}]$ for any $\kQ$-module $M$ where $\u X_M$ denotes the coefficient-free cluster character associated to $M$. 

	Note that the ring homomorphism $\Z[\y,\x^{\pm 1}] \fl \Z[\x^{\pm 1}]$ sending all the $y_i$'s to 1 sends $X_M$ to $\u X_M$ for any $\kQ$-module $M$ and it preserves the positive cones so that whenever $X_M \in \Z_{\geq 0}[\y,\x^{\pm 1}]$ we get $\u X_M \in \Z_{\geq 0}[\x^{\pm 1}]$. 

	Since $X_{M \oplus N}=X_M X_N$ for any two $\kQ$-modules $M$ and $N$, it is enough to consider the case where $M$ is indecomposable. The proof will consider several cases depending on the component of the Auslander-Reiten quiver $\Gamma(\modd \kQ)$ of mod-$\kQ$ containing $M$. 

	If $M$ is not in a regular component, then it is rigid so that $\chi(\Gr_{\e}(M)) \geq 0$ and thus $X_M \in \Z_{\geq 0}[\y,\x^{\pm 1}]$. Hence, we only need to treat the case where $M$ is an indecomposable regular module. 

	\subsection{Exceptional tubes}
		\begin{lem}
			Let $Q$ be an affine quiver and let $M$ be an indecomposable regular $\kQ$-module contained in an exceptional tube of $\Gamma(\modd\kQ)$. Then $X_M \in \Z_{\geq 0}[\y,\x^{\pm 1}]$.
		\end{lem}
		\begin{proof}
			The proof is similar to the proof of Theorem \ref{theorem:wild}. We denote by $R_i$ with $i \in \Z/p\Z$ the quasi-simple objects in the exceptional tube $\mathcal T$ containing $M$ ordered in such a way that $\tau R_i \simeq R_{i-1}$ for any $i \in \Z/p\Z$. Without loss of generality, we can assume that there exists some $n \geq 1$ such that $M \simeq R_1^{(n)}$. Then, as in the proof of Theorem \ref{theorem:wild}, we get $$X_M \in \Z_{\geq 0}[\y^{\ddim R_i}, \tau_j^{\pm 1}, \nu_i | 1 \leq i \leq n, 1 \leq j \leq n+1]$$ where for any $i \in \Z/p\Z$ we have set 
			$$\tau_i = L(R_i,0) \textrm{ and } \nu_i = \sum_{\mathbf e \neq \ddim R_i, 0} L(R_i,\mathbf e).$$
			Now, since the tube $\mathcal T$ is exceptional, it follows that its quasi-simple objects are rigid in mod-$\kQ$ and in particular each $R_i$ is rigid. Then $\nu_i \in \Z[\y,\x^{\pm 1}]$ for every $i \in \Z/p\Z$ and thus $X_M \in \Z[\y,\x^{\pm 1}]$.
		\end{proof}

	\subsection{Homogeneous tubes}
		For modules in homogeneous tubes, the proof will be divided into two steps. 

		We recall that for any $n \geq 1$, the \emph{$n$-th normalised Chebyshev polynomial of the first kind} is the polynomial $F_n \in \Z[x]$ defined by 
		$$F_0(x) = 2, \ F_1(x) = x \textrm{ and } F_{n+1}(x)=xF_n(x)-F_{n-1}(x)\textrm{ for }n \geq 1.$$

		\begin{lem}\label{lem:Fnpos}
			Let $Q$ be an affine quiver and let $M$ be an indecomposable $\kQ$-module which is quasi-simple in a homogeneous tube. Then for any $n \geq 1$ we have $$F_n(\u X_M) \in \Z[\x^{\pm 1}]$$
		\end{lem}
		\begin{proof}
			If $Q$ is the Kronecker quiver, this was proved in \cite{shermanz} so that we may assume that $Q$ has at least three vertices and thus, the Auslander-Reiten quiver of mod-$\kQ$ contains an exceptional tube $\mathcal T$ of rank $p >1$ whose quasi-simple modules, denoted by $R_i$ with $i \in \Z/p\Z$, are ordered in such a way that $\tau R_i \simeq R_{i-1}$ for any $i \in \Z/p\Z$.

			Then it follows from \cite[Proposition 3.3]{Dupont:transverse} that 
			$$F_n(\u X_M) = \u \Delta_{n,p}(\u X_{R_1}, \ldots, \u X_{R_{np}})$$
			for any $n \geq 1$.

			For any $i \in \Z/p\Z$ and for any dimension vector $\e=(e_i)_{i \in Q_0} \in \Z_{\geq 0}^{Q_0}$, we set 
			$$\u L(R_i, {\e}) = \chi(\Gr_{\e}(M)) \prod_{j \in Q_0} x_j^{-\<\e, S_j \> - \<S_j, \ddim R_i - \e\>}$$
			and
			$$\tau_i = \u L(R_i,0) \textrm{ and } \nu_i = \sum_{\mathbf e \neq \ddim R_i, 0} \u L(R_i,\mathbf e).$$

			We thus get 
			$$F_n(\u X_M) = \u \Delta_{n,p}(\tau_1+\nu_1+\frac{1}{\tau_2}, \cdots, \tau_{np}+\nu_{np}+\frac{1}{\tau_1})$$
			which belongs to $\Z_{\geq 0}[\tau_i^{\pm 1}, \nu_i | 1 \leq i \leq np]$ by Lemma \ref{lem:positivitedeltap}. Now, for any $i \in \Z/p\Z$, the module $R_i$ is quasi-simple in an exceptional tube so that it is rigid and thus each $\tau_i^{\pm 1}$ and each $\nu_i$ belongs to $\Z[\x^{\pm 1}]$. It follows that $F_n(\u X_M) \in \Z[\x^{\pm 1}]$.
		\end{proof}

		\begin{lem}\label{lem:tame}
			Let $Q$ be an affine quiver and let $M$ be an indecomposable $\kQ$-module in a homogeneous tube. Then $\u X_M \in \Z[\x^{\pm 1}]$.
		\end{lem}
		\begin{proof}
			Let $\mathcal T$ be a homogeneous tube and let $M$ be the quasi-simple module in $\mathcal T$. For any $n \geq 1$ we denote by $M^{(n)}$ the unique indecomposable module in $\mathcal T$ with quasi-length $n$. We need to prove that $\u X_{M^{(n)}} \in \Z_{\geq 0}[\x^{\pm 1}]$ for any $n \geq 1$.

			For any $n \geq 1$ we denote by $S_n \in \Z[x]$ the \emph{$n$-th Chebyshev polynomial of the second kind} given by $S_0(x)=1$, $S_1(x)=x$ and $S_{n+1}(x)=xS_n(x)-S_{n-1}(x)$ for any $n \geq 1$ or equivalently by $S_n(x)=P_n(1, \ldots, 1, x, \ldots, x)$. 

			It follows from \cite[Theorem 2]{Dupont:qChebyshev} (see also \cite{CZ}) that for any $n \geq 1$, we have $\u X_{M^{(n)}} = S_n(\u X_M)$. It is well-known that $S_n = \sum_{k=0}^{\lfloor n/2 \rfloor} F_{n-2k}$ so that 
			$$\u X_{M^{(n)}} = \sum_{k=0}^{\lfloor n/2 \rfloor} F_{n-2k}(\u X_M)$$
			and it thus follows from Lemma \ref{lem:Fnpos} that each of the terms in this sum belongs to $\Z[\x^{\pm 1}]$. This proves the lemma and Theorem \ref{theorem:tame} follows.
		\end{proof}
	
\section{Application to cluster bases in cluster algebras}\label{section:bases}
	Let $Q$ be an acyclic quiver and let $\mathcal A(Q,\x)$ be the corresponding (coefficient-free) cluster algebra. If $Q$ is an affine quiver, one can naturally construct three distinct cluster bases of particular interest in $\mathcal A(Q,\x)$ using the representation theory of $Q$. All these three constructions involve an element of $\mathcal A(Q,\x)$, which we we denote by $X_\delta$ and which is given by the value of $\u X_M$ on any quasi-simple module in a homogeneous tube. Note that this value is independent on the chosen homogeneous tube \cite[Lemma 5.3]{Dupont:genericvariables}.

	The first basis, initially constructed by Sherman and Zelevinsky for the Kronecker quiver \cite{shermanz} and by Cerulli for quivers of type $\widetilde A_{2,1}$ \cite{Cerulli:A21} is given by
	$$\mathcal B_Q = \mathcal M_Q \sqcup \ens{F_n(X_\delta)X_R| n \geq 1, R \textrm{ is a regular rigid $\kQ$-module}}.$$
	It is conjectured in \cite{Dupont:transverse} that this cluster linear basis of $\mathcal A(Q,\x)$ is actually \emph{canonically positive} (or \emph{atomic}) in the sense of \cite{Cerulli:A21}, that is, it is constituted of the extremal elements in the positive cone of $\mathcal A(Q,\x)$.

	The second basis was initially constructed by Caldero and Zelevinsky for the Kronecker quiver \cite{CZ}. It is given by
	$$\mathcal B_Q = \mathcal M_Q \sqcup \ens{S_n(X_\delta)X_R| n \geq 1, R \textrm{ is a regular rigid $\kQ$-module}}$$
	where for any $n \geq 1$, $S_n$ is the $n$-th normalised Chebyshev polynomial of the second kind whose definition is recalled in the proof of Lemma \ref{lem:tame}.
	
	The third basis was considered in \cite{Dupont:genericvariables} and, in a broader context by Geiss, Leclerc and Schr\"oer \cite{GLS:generic}. It is given by
	$$\mathcal B_Q = \mathcal M_Q \sqcup \ens{X_\delta^nX_R| n \geq 1, R \textrm{ is a regular rigid $\kQ$-module}}.$$
	and it is called the \emph{generic basis} of $\mathcal A(Q,\x)$.

	The fact that all these three bases are indeed cluster bases follows from \cite{Dupont:genericvariables,Dupont:BaseAaffine,DXX:basesv3} and more generally \cite{GLS:generic}. We conjecture that these three cluster bases consist of positive elements of $\mathcal A(Q,\x)$. As a byproduct of the methods used in this article, we can partially prove this conjecture. More precisely, we prove Theorem \ref{theorem:bases}, stating that the elements in these three bases can be written as subtraction-free Laurent polynomials in the initial cluster of the cluster algebra $\mathcal A(Q,\x)$.

	\subsection{Proof of Theorem \ref{theorem:bases}}
		It is well-known that for any indeterminate $z$, the element $z^n$ can be written as a $\Z_{\geq 0}$-linear combination of $S_k(z)$ with $k \leq n$ and $S_n(z)$ can be written as a $\Z_{\geq 0}$-linear combination of $F_k(z)$ with $k \leq n$ (see for instance \cite[Section 6]{Dupont:genericvariables}). Thus, any element in $\mathcal G_Q$ can be expressed as a $\Z_{\geq 0}$-linear combination of elements of $\mathcal C_Q$ and each element of $\mathcal C_Q$ can be expressed as a $\Z_{\geq 0}$-linear combination of elements of $\mathcal B_Q$. Thus, it is enough to prove that $\mathcal B_Q \subset \Z_{\geq 0}[x_1^{\pm 1}, \ldots, x_m^{\pm 1}]$ in order to prove the theorem.

		Elements in $\mathcal B_Q$ are either cluster monomials or products of the form $\u X_R F_n(\u X_M)$ where $n \geq 1$ and $R$ is a regular rigid $\kQ$-module and $M$ is a quasi-simple module in a homogeneous tube. According to \cite{Nakajima:cluster,Qin}, every cluster monomial and every element of the form $\u X_R$ where $R$ is regular rigid belongs to the cone $\Z_{\geq 0}[\x^{\pm 1}]$. Moreover, it follows from Lemma \ref{lem:Fnpos} that $F_n(\u X_M) \in \Z_{\geq 0}[\x^{\pm 1}]$. Since $\Z_{\geq 0}[\x^{\pm 1}]$ is stable under product, it follows that $\mathcal B_Q \subset \Z[\x^{\pm 1}]$ and Theorem \ref{theorem:bases} is proved. \qed

\section*{Acknowledgements}
	This paper was written while the author was at the university of Sherbrooke as a CRM-ISM postdoctoral fellow under the supervision of Ibrahim Assem, Thomas Br\"ustle and Virginie Charette. He would like to thank Giovanni Cerulli Irelli and Yann Palu for valuable discussions and comments on the topic.


\newcommand{\etalchar}[1]{$^{#1}$}
\providecommand{\bysame}{\leavevmode\hbox to3em{\hrulefill}\thinspace}
\providecommand{\MR}{\relax\ifhmode\unskip\space\fi MR }
\providecommand{\MRhref}[2]{%
  \href{http://www.ams.org/mathscinet-getitem?mr=#1}{#2}
}
\providecommand{\href}[2]{#2}

\end{document}